\documentclass[11pt]{article}

\usepackage{enumerate}
\usepackage{latexsym} %Des symboles de maths en plus.
\usepackage{theorem}
\usepackage{graphics}
\usepackage{graphicx}
\usepackage{amsmath}
\usepackage{amssymb}  %requis pour les polices mathbb
\usepackage[english]{babel} 
\usepackage{comment}
\usepackage{color}
\usepackage{tikz}
\usetikzlibrary{arrows}
\usepackage{ifpdf}

\usepackage{amsfonts}
\usepackage{setspace}
\usepackage{fullpage}
\usepackage{enumitem}
\usepackage{hyperref}

\newtheorem{theorem}{Theorem}

\newtheorem{remark}{Remark}

\newtheorem{lemma}{Lemma}

\newtheorem{conjecture}{Conjecture}

\newtheorem{case}{Case}

\newcommand{\ls}{\mathcal{L}}
\newcommand{\As}{\mathcal{A}}
\newcommand{\Ai}[1][i]{A_{#1}}
\newcommand{\Cs}{C}
\newcommand{\Ci}[1][i]{c_{#1}}
\newcommand{\hgt}{H}

\newcommand{\m}{\mathcal}
\newcommand{\abs}[1]{\left\lvert{#1}\right\rvert}
\newcommand{\mc}{\mathcal}

\newcounter{claim}

\newenvironment{proof}[1][]%
 {\noindent {\setcounter{claim}{0}\sc proof ---
   }{#1}{}}{\hfill$\Box$\vspace{2ex}}

	{\noindent {}{#1}{}}{ This proves Claim~\arabic{claim}.\vspace{1ex}}

\newcommand{\sm}{\setminus} %prive de 

\usepackage{ifpdf}

\ifpdf
\fi
\newcommand{\ov}{\overline}

\usepackage{microtype}
\usepackage{enumerate}

\title{De Bruijn-Erd\H{o}s type theorems for graphs and posets}
\author{
Pierre Aboulker\thanks{Universidad Andres Bello, Santiago, Chile, e-mail: pierreaboulker@gmail.com. Research supported by Fondecyt Postdoctoral grant 3150314 of CONICYT Chile.}
  \and 
Guillaume Lagarde\thanks{ENS Lyon, Lyon, France, e-mail: guillaume.lagarde@gmail.com} 
 \and 
David Malec\thanks{University of Maryland, College Park, United States, e-mail: dmalec@cs.wisc.edu}
 \and
Abhishek Methuku\thanks{Central European University, Budapest, Hungary, e-mail: abhishekmethuku@gmail.com} \and 
Casey Tompkins \thanks{Central European University, Budapest, Hungary, e-mail: ctompkins496@gmail.com}
}

\begin{document}

\maketitle

\begin{abstract} 
A classical theorem of De~Bruijn and Erd\H{o}s asserts that any noncollinear set of $n$ points in the plane determines at least $n$ distinct lines. 
We prove that an analogue of this theorem holds for graphs. Restricting our attention to comparability graphs, we obtain a version of the De~Bruijn-Erd\H{o}s theorem for partially ordered sets (posets). Moreover, in this case, we have an improved bound on the number of lines depending on the height of the poset. The extremal configurations are also determined. 
\end{abstract}

\section{Introduction}

The starting point of this paper is a classical result of De~Bruijn and Erd\H{o}s in combinatorial geometry.
%The following Theorem is known as the \emph{de Bruijn-Erd\H{o}s Theorem}.
A set of  $n$ points in the plane, not all on a line, is called a \textit{near-pencil} if exactly $n-1$ of the points are collinear. 

\begin{theorem} 
\label{dbe}
Every noncollinear set of $n$ points in the plane determines at least $n$ lines. Moreover, equality occurs if and only if the configuration is a near-pencil.
\end{theorem}

Erd\H{os} \cite{E43} showed that this result is a consequence of the Sylvester-Gallai theorem which asserts that every noncollinear set of $n$ points in the plane determines a line containing precisely two points. 
Later, De Bruijn and Erd\H{o}s~\cite{dbe} proved a more general combinatorial result which implies Theorem~\ref{dbe}.
Coxeter \cite{Cox} showed that the Sylvester-Gallai theorem holds in a more basic setting known as  \textit{ordered geometry}. Here, the notions of distance and angle are not used and, instead, a ternary relation of \textit{betweenness} is employed.
We write $[abc]$ for the statement that $b$ lies between $a$ and $c$.  Using this notation, a \textit{line} can be defined in a simple way. For any two distinct points $a$ and $b$, the line $\ov{ab}$ is defined as
\begin{equation} \label{definitionofline}
\ov{ab}=\{a,b\} \cup \{ x:[xab]\text{ or } [axb] \text{ or } [abx]\}
\end{equation}

%In the same paper, they proved that this result holds in a more abstract setting called \emph{linear space}. A linear space is a finite set of points along with a collection of subsets of those points, called lines, satisfying that every pair of points is contained in exactly one line. However, in this setting, additional extremal examples are provided by finite projective planes.

%Recently, Chv\'atal and his collaborators have obtained generalizations of Theorem \ref{dbe} in a new direction. 
%They consider an abstract setting called \emph{ordered geometry}~\cite{Cox}, which is a form of geometry whose only primitive notions are points and the ternary relation of \emph{betweenness} defined on these points.  For three distinct points $a$, $b$, $x$, we write $[axb]$ (following Menger \cite{Menger}) to denote that the point $x$ lies \emph{between} points $a$ and $b$ and we say that $\{a,b,x\}$ is collinear. 

Therefore, any set of points with a betweenness relation defined on them determines a family of lines. For example, in an arbitrary metric space $(V,d)$, Menger \cite{Menger} defined betweenness on $V$ in the following natural way:
\[
[axb]  \Leftrightarrow d(a,x) + d(x,b) = d(a,b).
\]
Using \eqref{definitionofline}, the line $\ov{ab}$ can be defined for any two distinct points $a$ and $b$ in $V$. Observe that this definition of a line generalizes the classical notion of a line in Euclidean space to any metric space. These lines may have strange properties: two lines might have more than one common point, and it is even possible for a line to be a proper subset of another line. A line is called \emph{universal} if it contains all points in $V$. Chen and Chv\'atal~\cite{CC} proposed the following conjecture, which, if true, would give a vast generalization of Theorem~\ref{dbe}.

\begin{conjecture}[Chen-Chv\'atal ~\cite{CC}] \label{conjCC}
Any finite metric space on $n$ points either induces at least $n$ distinct  lines or contains a universal line.
\end{conjecture}

Although the conjecture has been proved in several special cases~\cite{dh, chordal,ChenGraph,12Metric,ChvatalMetric,L1}, it is still wide open. The best known lower bound on the number of lines in a general finite metric space with no universal line is $(1/\sqrt 2 + o(1))\sqrt{n}$~\cite{metricSpace}. 
\medskip

Recently, Chen and Chv\'atal \cite{CC} generalized the notion of lines in metric spaces to lines in hypergraphs.  A \emph{hypergraph} is an ordered pair $(V, \mathcal{E})$ such that $V$ is a set of elements called the \emph{vertices} and $ \mathcal{E}$ is a family of subsets of $V$ called the \emph{edges}. A hypergraph is \emph{k-uniform} if each of its edges consist of $k$ vertices. They observed that given a metric space $(V,d)$, one can associate a hypergraph $H(d) = (V, \mathcal E)$ with $\mathcal E := \{ \{a, b, c\} : [abc] \text{ in }(V,d)\}$. If the line $\ov{ab}$ in the $3$-uniform hypergraph is defined as
\begin{equation*} \label{hypeLine}
\ov{ab}=\{a,b\} \cup \{x: \{a,b,x\} \in \mathcal{E}  \},
\end{equation*}
then the metric space $(V,d)$ and the hypergraph $(V, \mathcal E)$ determine the same set of lines.
 
They proved that there is an infinite family of $3$-uniform hypergraphs inducing only $c^{\sqrt{\log_2{n}}}$ distinct lines (where $n$ is the number of vertices and $c$ is a constant). This means that there are infinitely many $3$-uniform hypergraphs for which the analogue of Theorem~\ref{dbe} does not hold. However, analogues of Theorem~\ref{dbe} have been shown to hold for some special families of $3$-uniform hypergraphs in \cite{BBC}. The best known lower bound on the number of lines in a $3$-uniform hypergraph with no universal line is $(2-o(1)) \log_2{n}$ \cite{hypergraphA}.
\medskip

Following the lead of these previous works, we obtain an analogue of De Bruijn-Erd\H{o}s theorem for posets. Let $P = (X, \prec)$ be a finite poset with the order relation $\prec$ defined on the set $X$.  The size of a maximum chain in $P$ is called the \textit{height} of $P$ and is denoted $h(P)$.

As in the metric space case, a poset $P$ induces a natural betweenness relation: 
\begin{displaymath}
[abc]  \Leftrightarrow a\prec b \prec c \text{ or } c \prec b \prec a.
\end{displaymath}
Therefore, we can again define lines in posets using \eqref{definitionofline}. Observe that, if $a$ is incomparable to $b$, then the line $\ov{ab}=\{a,b\}$ and, if $a$ is comparable to $b$, then
\begin{displaymath}
 \overline{ab} = \{a,b\} \cup \{x: \mbox{$x$ is comparable to both } a \mbox{ and } b\}.
\end{displaymath}
As before, a line is \emph{universal} if it contains every point from the ground set.  Our main result is to show that an analogue of Conjecture~\ref{conjCC}  holds for posets. In fact, we obtain a stronger bound as a function of the height of the poset. 

\begin{theorem}
\label{lineshofP}
If $P$ is a poset on $n$ vertices with no universal line, and $h(P) \ge 2$, then $P$ induces at least
\begin{equation}
\label{eq:fnofH}
h(P) \binom{\lfloor n/h(P) \rfloor}{2} + \lfloor n/h(P) \rfloor (n \bmod {h(P)}) + h(P)
\end{equation}
distinct lines.
\end{theorem}

Observe that \eqref{eq:fnofH} is always greater than or equal to $n$ (with equality if $h(P) \geq \lfloor n/2 \rfloor$). 
Moreover, if $h(P) = O(n^s)$ for $0<s\leq 1$, the number of distinct lines in $P$ is $\Omega(n^{2-s})$.
\medskip

Our second result is a generalization from posets to graphs.   For any graph $G=(V,E)$ and vertices $a,b \in V$, we can define the line $\ov{ab}$ as
\begin{equation*}
\ov{ab}=\{a,b\} \cup \{c: abc \text{ is a triangle}\}.
\end{equation*}
Again, the line $\ov{ab}$ is universal if it contains every vertex in $V$.  We prove

\begin{theorem}\label{mainGraph}
If a graph $G$ on $n \ge 4$ vertices does not contain a universal line, then  it induces at least $n$ distinct lines, and equality occurs only if $G$ consists of a clique of size $n-1$ and a vertex that has at most one neighbor in the clique.
\end{theorem}

\begin{remark}
It may be easily seen that the theorem also holds when $n=3$, but we have an additional extremal example in this case: a graph where all pairs of vertices are non-adjacent.
\end{remark}

Observe that to any poset $P=(V, \prec)$, we can associate a graph  $G=(V,E)$ where $ab \in E$ if and only if $a\prec b$ or $b \prec a$. Such a graph is called a \emph{comparability graph}. Hence, for any three vertices $a,b,c$ of $P$, we have  $a \prec b \prec c$ if and only if $abc$ is a triangle in the corresponding comparability graph. Therefore, the graph case is a generalization of the poset case. 

Using Theorem \ref{mainGraph}, it can be easily seen that, in the case of posets, equality occurs only if the poset consists of a chain of size $n-1$ and a vertex which is comparable to at most one vertex of this chain.

The paper is organized as follows.  In Section~\ref{poset}, we prove Theorem \ref{lineshofP}  by providing an algorithm for finding lines. In Section~\ref{graph}, we prove  Theorem \ref{mainGraph} by induction.

\section{Lines in Posets}\label{poset}
We begin by introducing some notation that will be useful in the proof of Theorem \ref{lineshofP}. For any pair of elements $a, b$ in a poset $P = (X, \prec)$, we write $a \nsim b$ to indicate that the points $a$ and $b$ are not comparable (that is, neither $a \prec b$ nor $b \prec a$ hold). Let $Y \subseteq X$. We denote by $P \sm Y$  the poset on the set of points $X \sm Y$  together with $\prec$ restricted to $X \sm Y$.

In this section, we prove a lower bound on the number of lines in a poset as a function of its height. Before we proceed with the proof, we need a simple lemma.

\begin{lemma}
\label{estimating}
If $A_1,...,A_{r}$ are $r$ sets such that $\sum_{i=1}^{r} |A_i| = n$, then $\sum_{i=1}^{r}\binom{\abs{\Ai}}{2}\ge r \binom{\lfloor n/r ) \rfloor}{2} + \lfloor n/r \rfloor (n\bmod {r})$.
\end{lemma}

\begin{proof}
Observe first that if $\abs{\abs{A_i}-\abs{A_j}}\le1$ for all $i$ and $j$, then the bound holds.   Thus, it suffices to prove that if $\abs{A_i}-\abs{A_j}> 1$, then moving one point from $A_i$ to $A_j$ does not increase $\sum_{i=1}^{\hgt}\binom{\abs{\Ai}}{2}$.  Let $x \in A_i \setminus A_j$ and $A_i' = A_i \setminus \{x\}$, $A_j' = A_j \cup \{x\}$.  We have
\begin{displaymath}
\binom{\abs{A_i}}{2}+\binom{\abs{A_j}}{2} \ge \binom{\abs{A_i}-1}{2}+ \binom{\abs{A_j}+1}{2} =\binom{\abs{A_i'}}{2}+\binom{|A_j'|}{2} 
\end{displaymath}
by the convexity of $\binom{m}{k}$ in $m$, and the lemma follows.
\end{proof}

\subsection{Proof of Theorem \ref{lineshofP}}
 Let $\As$ be a maximal partition of $P$ into antichains, and let $\Cs \subseteq P$ be
  a maximal chain in $P$.  By Mirsky's Theorem, we know that $\abs{\As}=\abs{\Cs} = h(P)$. For notational convenience, from now on, let $h(P)$ be denoted by $\hgt$. Denote the elements of $\As$ and $\Cs$,
  respectively, as $\As=\{\Ai[1],\dots,\Ai[\hgt]\}$ and
  $\Cs=\{\Ci[1] \dots \Ci[\hgt]\}$ with $\Ci[1]\prec\dots\prec\Ci[\hgt]$.  
Assume, without loss of generality, that $\Ci\in\Ai$ for  $i=1, \dots, H$.  
%This assumption is without loss   of generality, since we know the elements of $\Cs$ are all   comparable and so must each be contained in distinct antichains from   $\As$.

  %Our proof proceeds by exhibiting a set $\ls(P)$ of our required number of distinct lines in $P$. We start by setting 
Set  
\begin{equation*}
   \mc L_0
    :=
    \mathop{\bigcup} \limits_{i=1}^{\hgt}\{\overline{ab}:a,b\in\Ai, a \neq b\}.
  \end{equation*}
Note that all of the lines in $\mc L_0$ are induced by incomparable points and are, thus, pairwise distinct. 
By Lemma \ref{estimating}, we have
  \begin{equation*}
    \abs{\ls_{0}}=\sum_{i=1}^{\hgt}\binom{\abs{\Ai}}{2}\ge \hgt \binom{\lfloor n/\hgt) \rfloor}{2} + \lfloor n/\hgt \rfloor (n\bmod {\hgt}).
  \end{equation*}
  
%\textbf{sketch of the proof of the inequality :} We prove that the worst case for the size of the sets $|A_i|$ is when they are well-balanced, i.e they have the size $\lfloor n/\hgt \rfloor$ or $\lceil n/\hgt \rceil$. If it is the case, the extremal case is the following : all sets have the size $\lfloor n/\hgt \rfloor$ except $(n mod H)$ of them who have $\lceil n/\hgt \rceil$ elements (to see why, add one by one the points in a well-balanced way). The result follows as a simple counting. \\
%Now, let us prove that the worst case is when sets are well-balanced: if it is not the case, we can find a set $|A_{i0}| < \lfloor n/\hgt \rfloor$ and a set $|A_{i1}| \geq \lceil n/\hgt \rceil$ (pigeonhole principle). Construct a new family of size $A'_i$ such that $|A'_{i0}| = |A'_{i0}| +1$, $|A'_{i1}| = |A'_{i1}| -1$ and $|A'_{j}| = |A_j|$ in the others cases. The sum of the size of the sets is still $n$ but it is easy to prove that the sum according to this new family is strictly less important than the sum according the first one. Continue this process until the family is well-balanced and the result is proved. \\ \\

 Next we use the chain $C$ to find $H$ further lines, distinct from those in $\ls_0$.  We do so
  via the following iterative process:
 
  \medskip

 Set $b_1=1$, $t_1=h$ and  $\m L_1= \emptyset$. 
 For $k=1, 2, \dots$, apply the following steps until a STOP condition is met.
\begin{enumerate}
\item[{\bf Step 1}] If $b_k=t_k$, set $\m L_k := \m L_{k-1} \cup  \{\ov{c_1c_H}\} $ and {\bf STOP}. 

Otherwise $b_k<t_k$ and there exists $s_k \notin \ov{c_{b_k}c_{t_k}}$. If $s_k$ is incomparable with both $c_{b_k}$ and $c_{t_k}$, go to {\bf Step 2a}. If $s_k$ is incomparable with  $c_{b_k}$ and comparable with $c_{t_k}$, go to {\bf Step 2b}. Finally, if $s_k$ is incomparable with $c_{t_k}$ and comparable with $c_{b_k}$, go to {\bf Step 2c}. 

\item[{\bf Step 2a}] Set $\m L_k := \m L_{k-1} \cup \{\ov{c_is_k} : b_{k}\le i \le t_k\} \cup  \{\ov{c_1c_H}\}$ and {\bf STOP}.

\item[{\bf Step 2b}] Set $b_{k+1}=max_{b_k\le i < t_k}\{i: \, c_i \nsim s_k \}+1$, 

$t_{k+1}=t_k$, 
$\m L_k:= \m L_{k-1} \cup \{\ov{c_is_k} : b_k\le i \le b_{k+1}\}$. 
Go to {\bf Step 1}.

Observe that, in this case, we have that $b_k<b_{k+1}\le t_k$, $s_k \prec c_{b_{k+1}}$ and, for $j=b_k,  \dots, b_{k+1}-1$, we have $s_k \nsim c_j$.

\item[{\bf Step 2c}] Set $t_{k+1}=min_{b_k< i \le t_k}\{i: \, c_i \nsim s_k \}-1$, 

$b_{k+1}=b_k$, 
$\m L_k:= \m L_{k-1} \cup \{\ov{c_is_k} : t_{k+1}\le i \le t_k\}$. 
Go to {\bf Step 1}.

Observe that, in this case, we have that $b_k \le t_{k+1} < t_k$,  $c_{t_{k+1}} \prec s_k$ and, for $j=t_{k+1}+1, \dots, t_{k}$, $s_k \nsim c_j$.

\end{enumerate}

Assume that the process stops after $K$ iterations. 
%So, the process creates two sequences $(b_k)_{1 \le k \le K}$  and $(t_k)_{1 \le k \le K}$ such that $b_1 \le \dots \le b_K \le t_K \le \dots \le t_1$. 

For any $k < K$, in the $k^{th}$ iteration, exactly one line added to $\mc L_{k}$ is induced by two comparable points.  We call this line $l_k$. Thus, there are $K-1$ such lines, namely $l_1, \dots, l_{K-1}$. Notice that $l_k$ is either $\ov{c_{b_{k+1}}s_k}$ or $\ov{c_{t_{k+1}}s_k}$. 
If $l_k=\ov{c_{b_{k+1}}s_k}$, since $s_k \prec c_{b_{k+1}}$, we have  $\{c_{b_{k+1}}, c_{b_{k+2}}, \dots c_{b_K}, c_{t_K}, c_{t_{K-1}}, \dots, c_{t_1} \} \subseteq \ov{c_{b_{k+1}}s_k}$,  and since $c_{b_k} \nsim s_k$, we have $c_{b_k} \notin   \ov{c_{b_{k+1}}s_k}$. 
Similarly, if $l_k=\ov{c_{t_{k+1}}s_k}$ we have $\{c_{b_1}, \dots, c_{b_K}, c_{t_K}, \dots, c_{t_{k+1}}\} \subseteq \ov{c_{t_{k+1}}s_k}$ and $c_{t_k} \notin \ov{c_{t_{k+1}}s_k}$. 
Observe now that the line $\ov{c_1c_H}$, that is added at the $K^{th}$ iteration, contains all points in $C$.
This implies that the lines  $l_1, \dots, l_{K-1}, \ov{c_1c_H}$ are pairwise distinct. Thus, the process finds $K$ pairwise distinct lines which are induced by comparable points. Moreover, since all the lines in $\mc L_0$ are induced by incomparable points, none of these $K$ lines belong to $\mc L_0$. 

The rest of the lines found by the process are induced by incomparable points. Hence, it remains to prove that $H-K$ of them are pairwise distinct and don't belong to $\mc L_0$.

Let $k<K$. 
We claim that $\mc L_k$ contains at least $b_{k+1}-b_k + t_k-t_{k+1}-1$ (new) lines that are not in $\mc L_{k-1}$. 
Assume first that, in the $k^{th}$ iteration,  lines are added at Step 2b (so $t_k-t_{k+1}=0$). So $b_{k+1} - b_{k}$ lines induced by two incomparable points are added, namely $\ov{c_{b_k}s_k}, \dots, \ov{c_{b_{k+1}-1}s_k}$. 
At most one of these lines belongs to $\mc L_0$ and none of them belongs to $\mc L_{k-1} \setminus \mc L_0$ because lines induced by incomparable points that are added in previous iterations of the process,  involve  points of $C$ either strictly below $b_k$ or  strictly above $t_k \ge b_{k+1}$. 
Hence, at least $b_{k+1} - b_{k}-1$ new  lines induced by incomparable points are added at Step 2b. A symmetric argument proves that, in the case where the lines are added at Step 2c (so $b_{k+1}-b_k=0$), we have added $t_k-t_{k+1}-1$ new lines induced by incomparable points.

So, after $K-1$ iterations, the number of  lines induced by incomparable points, in $\mc L_{K-1} \sm \mc L_0$ is 
\begin{displaymath}
\sum_{k=1}^{K-1} (b_{k+1}-b_{k} + t_{k}-t_{k+1} -1)= t_1-b_1 -( t_K-b_K) -(K-1)
= H-K-(t_K-b_K).
\end{displaymath}

Hence, it remains to show that $t_K-b_K$ new distinct lines induced by incomparable points are added at the $K^{th}$ iteration. 
In the case where $b_K=t_K$ we are done, so we may assume that $b_K<t_K$ and the process terminates at Step 2a. 
So, the lines $\ov{c_{b_K}s_k}, \,  \ov{c_{b_K+1}s_k}, \dots, \, \ov{c_{t_K}s_k}$ are added.
At most one of these lines belong to $\mc L_0$ and none of them belong to $\mc L_{K-1} \setminus \mc L_0$ since lines induced by incomparable points added at iterations $1, \dots, K-1$  involve  points of $C$ either strictly below $b_K$ or  strictly above $t_K$. 
It follows that $t_K-b_K$ new lines are added.

\section{Lines in Graphs}\label{graph}

We first need two easy observations about lines in a graph $G=(V,E)$. 
A vertex $x$ in a graph is \emph{universal} if it is adjacent to all vertices in $V \sm x$.  

\begin{enumerate}
\item If $ab \notin E$, then $\ov{ab}=\{a,b\}$,
\item A line $\ov{ab}$ is universal if and only if both $a$ and $b$ are universal.
\end{enumerate}
We are now ready to prove our generalization to the graph case.

\subsection{Proof of Theorem \ref{mainGraph}}

We will use induction on $n$ on the full statement of the theorem. First, we show that the theorem holds when $n=4$ . If there is no triangle in our graph, then every pair of vertices induces a distinct line, giving us $6$ lines. If there are two different triangles in our graph, then there exist $2$ vertices $p, q$, that belong to both triangles, and the line $\ov{pq}$ is universal, a contradiction. Therefore, we have exactly one triangle, and it is easy to see that in this case we have exactly $4$ lines in our graph and the extremal graphs are exactly as desired.

Let $G=(V, E)$ be a graph on $n \ge 5$ vertices having no universal lines, and assume the statement holds for smaller $n$.  

Let $V_1 \subseteq V$ be the set of those points $x$ such that $G \setminus \{x\}$  has a universal line, and set $V_2 = V \setminus V_1$. Assume first that $V_2= \emptyset$. So $V_1=V$ and, thus, for any $x \in V$, $V \sm \{x\}$, induces a universal line. Since $G$ has no universal lines, $V \sm \{x\}$ is  a line of $G$ for any $x \in V$. Thus, $G$ induces $n$ distinct lines of size $n-1$. 
Moreover, since it has no universal lines, $G$ has at least two non-adjacent vertices, providing a line of size two. Thus, if $V_2 = \emptyset$, we have that $G$ induces at least $n+1$ lines. 

So we may assume from now on that $V_2 \neq \emptyset$. We will distinguish between two cases:

\begin{case}
There exists a point $x$ in $V_2$ that is not universal.
\end{case}

Let $y$ be a vertex non-adjacent to $x$. 
Since $G \sm \{x\}$ has no universal lines, by induction, $G \sm \{x\}$ induces at least $n-1$ distinct lines. If $\ell$ is a line of $G \setminus \{x\}$, then either $\ell$ or $\ell \cup \{x\}$ is a line of $G$. It follows that these lines are all distinct in $G$. Moreover, if they contain $x$, then they have at least $3$ vertices and so they are all distinct from $\ov{xy}=\{x,y\}$. Hence, $G$ has at least $n$ distinct lines.

Now, assume that $G$ induces exactly $n$ distinct lines. Then, $G \setminus \{x\}$ must contain exactly $n-1$ lines and so by induction, $G \setminus \{x\}$ consists of a clique $K$ on $n-2$ vertices, $x_1, x_2, \dots, x_{n-2}$, and a vertex $z$ which has at most one neighbor in $K$. Notice that the set of lines of $G \sm \{x\}$ is $ \mathcal{L}_{G \setminus \{x\}}:= \{\{x_1, \dots, x_{n-2}\}, \{z,x_1\}, \dots, \{z,x_{n-2}\}\}$, giving us $n-1$ distinct lines of $G$, namely, $\mathcal{L} := \{\ell \text{ or } \ell \cup \{x\} \mid \ell \in \mathcal{L}_{G \setminus \{x\}} \}$. 

We claim that $x$ is adjacent to all vertices of $V \sm \{x,y\}$ because otherwise there exists a vertex $y'$ in $V \sm \{x,y\}$ such that $\ov{xy'}=\{x, y'\}$ is a line of $G$. Since $\ov{xy}=\{x,y\}$ is also a line of $G$, and $\ov{xy}, \ov{xy'} \not \in \mathcal{L}$, $G$ induces at least $n+1$ distinct lines contradicting our assumption.

Assume that $y \in K$, and let $z'$ be the unique neighbor of $z$ in $K$. Consider a vertex $w$ in $K \sm \{y, z'\}$ (such a vertex exists because $n \ge 5$). Since $y \not \in \ov{xw}$ and $z \not \in \ov{xw}$, we have $\ov{xw} \not \in \mathcal{L}$. Of course, $\ov{xy} \not \in \mathcal{L}$ is a line of $G$ like before. Thus, $G$ induces at least $n+1$ distinct lines again. Therefore, $y=z$, and $K \cup \{x\}$ is a clique of $G$ as desired. 

\begin{case}
All points of $V_2$ are universal.
\end{case}

Since $G$ has no universal line, if follows that $V_2$ contains exactly one vertex, say $x$. So $V_1 = V \sm \{x\}$ and, thus, for any $u \in V_1$, $V \sm \{u\}$, is a line of $G$. This yields $n-1$ lines of size $n-1$. 
Moreover, it is easy to see that, since $G$ has no universal lines, it must contain at least two pairs of non-adjacent vertices, providing us with two more distinct lines of size two. Hence, $G$ has at least $n+1$ distinct lines.

\end{document}